\documentclass[amsthm,twocolumn]{autart_dateremoved}
\usepackage{amsmath,amssymb,amsfonts}
\usepackage{multirow}
\usepackage{enumerate}
\usepackage{graphicx}
\usepackage[numbers,sort&compress]{natbib}

\newtheorem{theorem}{Theorem}
\newtheorem{definition}{Definition}
\newtheorem{proposition}{Proposition}
\newtheorem{corollary}{Corollary}
\newtheorem{example}{Example}
\newtheorem{remark}{Remark}
\newtheorem{problem}{Problem}

\begin{document}

\begin{frontmatter}

\title{On quotients of Boolean control networks}

\author[dlut]{Rui~Li}\ead{rui\_li@dlut.edu.cn},
\author[uibe]{Qi~Zhang}\ead{zhangqi@uibe.edu.cn},
\author[pku]{Tianguang~Chu}\ead{chutg@pku.edu.cn}

\address[dlut]{School of Mathematical Sciences, Dalian University
of Technology, Dalian 116024, China}
\address[uibe]{School of Information Technology \& Management,
University of International Business \& Economics, Beijing 100029,
China}
\address[pku]{State Key Laboratory for Turbulence and Complex Systems,
College of Engineering, Peking University, Beijing 100871, China}

\begin{abstract}
In this paper, we focus on the study of quotients of Boolean control
networks (BCNs) with the motivation that they might serve as smaller
models that still carry enough information about the original
network. Given a BCN and an equivalence relation on the state set,
we consider a labeled transition system that is generated by the
BCN. The resulting quotient transition system then naturally
captures the quotient dynamics of the BCN concerned. We therefore
develop a method for constructing a Boolean system that behaves
equivalently to the resulting quotient transition system. The use of
the obtained quotient system for control design is discussed and we
show that for BCNs, controller synthesis can be done by first
designing a controller for a quotient and subsequently lifting it to
the original model. We finally demonstrate the applicability of the
proposed techniques on a biological example.
\end{abstract}


\end{frontmatter}

\section{Introduction}

\label{sec1}

Boolean networks (BNs) and Boolean control networks (BCNs), wherein
each component is characterized with a binary variable, have been
widely employed in modeling biological regulatory networks. After
assembling the components of a system as well as their regulatory
interactions, BN/BCN models can nicely describe the qualitative
temporal behavior of the system \cite{albert2014}. They can also
illuminate how perturbations may disrupt normal behavior and yield
testable predictions which are particularly valuable in less well
understood biological systems \cite{assmann2009}. As a nice
framework for modeling dynamical processes on networks, especially
in biological context, BN/BCN models have led to fruitful insights
for unicellular organisms \cite{christensen2009}, plants
\cite{akman2012}, animals \cite{chaves2008}, and humans
\cite{schlatter2011}, especially human signaling networks implicated
in diseases \cite{zhang2008}. A BN/BCN is typically placed in the
form of a nonlinear (control) system; while interestingly, based on
an algebraic state representation approach, the Boolean dynamics can
be mapped exactly into a standard discrete-time linear dynamics
\cite{chengbook}. This formal simplicity makes it relatively easy to
formulate and solve classical control-theoretic problems for
BNs/BCNs, and thereby lays a suitable foundation for a series of
subsequent studies. Examples include recent investigations of
dynamical properties \cite{hochma2013}, network synchronization
\cite{zhong2014}, controllability \cite{liang2017,lu2016b} and
stabilizability \cite{bof2015,li2017b,li2017,meng2019},
observability \cite{cheng2016,yu2020,yu2019b,zhou2019} and
reconstructibility \cite{fornasini2013a}, disturbance or
input-output decoupling \cite{liu2017,valcher2017,yu2019}, optimal
control \cite{fornasini2014,wu2019}, and more
\cite{rafimanzelat2019,wang2019,zhang2015,zhang2018,zou2015}. The
size of the linear system that describes a BN with $n$ state
variables is $2^n$. Thus, any algorithm based on this algebraic
set-up has an exponential time complexity in the worst case. On the
other hand, it has shown that for several control problems, the
complexity curse can be alleviated or even removed if the structure
of BNs is appropriately constrained \cite{gao2018b,weiss2019}. These
positive developments notwithstanding, it still seems
computationally challenging to solve control-related problems in
general BNs/BCNs, since many such problems have shown to be NP-hard
\cite{akutsu2007,laschov2014,laschov2013b,zhang2016b}. The hardness
results justify the use of exponential time algorithms and
exponential size systems suggested by the algebraic state-space
representation.

In this paper, we focus on studying quotients of BCNs since they can
be seen as lower dimensional models that may still contain enough
information about the original model (whose algebraic representation
is of exponential size). We consider quotient systems for BCNs in
the exact sense that the notion is used in the control community
\cite{chutinan2001,tabuada2005,tabuadabook}. Precisely, given a BCN
and an equivalence relation on its state set, we consider a
(labeled) transition system generated by the BCN and partition the
state set based on the relation. The resulting quotient system then
naturally captures the quotient dynamics of this BCN, so we propose
to develop a Boolean system that generates the transitions of the
quotient transition system (Theorem~\ref{thm1}). Of course, it is
not surprising that additional constraints need to be placed on the
equivalence relation to ensure that the quotient dynamics can indeed
be generated from some Boolean system. A subsequent question is
then, how to obtain an equivalence relation which allows the
construction of a quotient BCN. We fully answer this question by
giving a procedure that converges in a finite number of iterations
to a satisfactory equivalence relation (Theorem~\ref{thm2}). As
applications of the study, we show how the resulting quotient can be
used for controller synthesis. The results tell us that synthesizing
controllers for a BCN can be easily done by first controlling the
quotient system and then lifting the control law back to the
original Boolean model (see Propositions~\ref{prop2} and
\ref{prop3}).

\textit{Notation.} The symbol $\delta_k^i$ denotes the $i$th
canonical basis vector of length $k$, $\Delta_k$ denotes the set
consisting of the canonical vectors $\delta_k^1, \ldots,
\delta_k^k$, and $\mathcal{L}^{k \times r}$ denotes the set of all
$k \times r$ matrices whose columns are canonical vectors of length
$k$. Elements of $\mathcal{L}^{k \times r}$ are called logical
matrices (of size $k \times r$). A $(0,1)$-matrix is a matrix that
consists solely of the $0$ and $1$ entries. The $(i,j)$-entry of a
matrix $A$ is invariably denoted by $(A)_{ij}$. If $A$ and $B$ are
$k \times r$ $(0,1)$-matrices, the meet of $A$ and $B$, denoted by
$A \wedge B$, is the $(0,1)$-matrix with the $(i,j)$-entry equal to
$(A)_{ij} \wedge (B)_{ij}$. For a $k \times l$ $(0,1)$-matrix $C$
and an $l \times r$ $(0,1)$-matrix $D$, the Boolean product of $C$
and $D$, denoted by $C \odot D$, is the $k \times r$ matrix with the
$(i,j)$-entry $\bigvee_{s=1}^l [(C)_{is} \wedge (D)_{sj}]$. Given
two relations $\mathcal{R}_1$ and $\mathcal{R}_2$, $\mathcal{R}_2
\circ \mathcal{R}_1$ denotes the composition of $\mathcal{R}_1$ and
$\mathcal{R}_2$, i.e., the relation defined by $(a, c) \in
\mathcal{R}_2 \circ \mathcal{R}_1$ if and only if there exists $b$
with $(a,b) \in \mathcal{R}_1$ and $(b,c) \in \mathcal{R}_2$.

\section{Preliminaries}

\label{sec2}

\subsection{Algebraic representation of Boolean control networks}

\label{sec2.1}

A BCN is a discrete-time dynamical system with binary state
variables and binary control variables, i.e.,
\begin{align} \label{eq1}
x_1(t+1) &= f_1 (x_1(t), \ldots, x_n(t), u_1(t), \ldots, u_m(t)) ,
\notag
\\
& \vdots \\
x_n(t+1) &= f_n (x_1(t), \ldots, x_n(t), u_1(t), \ldots, u_m(t)) ,
\notag
\end{align}
with $x_i, u_j \in \{1,0\}$ and $f_i \colon \{1,0\}^{n+m}
\rightarrow \{1,0\}$. The dynamics (\ref{eq1}) can be recast into a
form similar to that of a discrete-time linear system, using the
semitensor product of matrices \cite{chengbook}. To be more precise,
we recall that the (\textit{left}) \textit{semitensor product} of
two matrices $A$ and $B$ of sizes $n_1 \times m_1$ and $n_2 \times
m_2$, respectively, denoted by $A \ltimes B$, is defined by $A
\ltimes B = (A \otimes I_{l/m_1}) (B \otimes I_{l/n_2})$, where
$\otimes$ is the Kronecker product of matrices, and $I_{l/m_1}$ and
$I_{l/n_2}$ are the identity matrices of orders $l/m_1$ and $l/n_2$,
respectively, with $l$ being the least common multiple of $m_1$ and
$n_2$. If we identify the Boolean values $1$ and $0$ with the
canonical vectors $\delta_2^1$ and $\delta_2^2$, respectively (so
$x_i$ and $u_j$ in (\ref{eq1}) are vectors in $\Delta_2$), and if we
let $x(t) = x_1(t) \ltimes \cdots \ltimes x_n(t)$ and $u(t) = u_1(t)
\ltimes \cdots \ltimes u_m(t)$, then the Boolean dynamics
(\ref{eq1}) can be represented by an equation of the form
\begin{equation} \label{eq2}
x(t+1) = F \ltimes u(t) \ltimes x(t) ,
\end{equation}
where $F \in \mathcal{L}^{2^n \times 2^{n+m}}$. (The expression on
the right-hand side of (\ref{eq2}) is unambiguous, since the
semitensor product is associative.) For more information on
converting a BCN in the form of (\ref{eq1}) to its algebraic
representation (\ref{eq2}), as well as more information regarding
the properties of the semitensor product, the reader is referred to,
e.g., \cite{chengbook} and \cite{cheng2009b}.

\subsection{Transition systems}

\label{sec2.2}

Our discussion of quotients of BCNs will be based on the notion of
quotient transition systems. We first recall the concept of a
(labeled) transition system.

\begin{definition}[See, e.g., \cite{tabuada2004}] \label{def1}
\normalfont
A (\textit{labeled}) \textit{transition system} is a
tuple $\mathcal{T} = (Q, L, \rightarrow)$ that consists of a set of
states $Q$, a set of labels $L$, and a transition relation
$\rightarrow \, \subseteq Q \times L \times Q$.
\end{definition}

For any $q, q' \in Q$ and any $l \in L$, a transition $(q, l, q')
\in \rightarrow$ means that it is possible to move from state $q$ to
state $q'$ under the action labeled by $l$. Following standard
practice, we denote $q \overset{l}\rightarrow q'$ if $(q, l, q') \in
\rightarrow$.

Recall that an equivalence relation $\mathcal{R}$ on $Q$ is a
reflexive, symmetric, and transitive binary relation on $Q$. Given a
transition system $\mathcal{T}$, if $\mathcal{R}$ is an equivalence
relation on the state set of $\mathcal{T}$, then it naturally
induces a quotient transition system, as follows.

\begin{definition}[See, e.g., \cite{tabuadabook}] \label{def2}
\normalfont
Let $\mathcal{T} = (Q, L, \rightarrow)$ be a transition
system and let $\mathcal{R}$ be an equivalence relation on $Q$. The
quotient transition system $\mathcal{T}/\mathcal{R}$ is defined by
$\mathcal{T} / \mathcal{R} = (Q/\mathcal{R}, L,
\rightarrow_{\mathcal{R}})$, where $Q/ \mathcal{R}$ is the quotient
set (i.e., the set of all equivalence classes $[q] = \{p \in Q
\colon (q,p) \in \mathcal{R} \}$ for $q \in Q$), and for all $[q],
[q'] \in Q/ \mathcal{R}$, $[q] \overset{l}\rightarrow_{\mathcal{R}}
[q']$ if and only if there exist $p \in [q]$ and $p' \in [q']$ such
that $p \overset{l}\rightarrow p'$.
\end{definition}

That is, a state $[q]$ in $\mathcal{T}/ \mathcal{R}$ can make a
transition to another state $[q']$ under an action $l$, if some $p
\in [q]$ can make a transition to some $p' \in [q']$ when taking the
action $l$. In what follows, we will use a similar framework to
study quotients of a BCN.

\section{Quotients of Boolean control networks}

\label{sec3}

\subsection{Constructing quotient Boolean systems}

\label{3.1}

Let us consider a BCN described by the algebraic representation
\begin{multline} \label{eq3}
\Sigma \colon \;\; x(t+1) = F \ltimes u(t) \ltimes x(t) , \quad
x\in \Delta_N, \quad u \in \Delta_M , \\
F \in \mathcal{L}^{N \times NM} .
\end{multline}
(Note that, in the above, $N$ and $M$ are in fact certain powers of
$2$, but we do not need this fact for our argument.) In order to
investigate quotients of (\ref{eq3}), we first turn our attention to
the equivalence relations on its state set $\Delta_N$. An immediate
observation is that every such equivalence relation $\mathcal{R}$
can be viewed as induced by a logical matrix $C$ with $N$ columns,
by saying
\begin{equation}\label{eq4}
(x, x') \in \mathcal{R} \Longleftrightarrow Cx = Cx' .
\end{equation}
Furthermore, the logical matrix $C$ can be chosen of full row rank
(hence in particular having no zero rows). We remark that such a
full row rank matrix can be directly derived from the \textit{matrix
representation} of $\mathcal{R}$. In fact, let $A_{\mathcal{R}}$ be
the $N \times N$ matrix whose entries are given by
\begin{equation*}
(A_{\mathcal{R}})_{ij} =
\begin{cases}
1 & \text{if $(\delta_N^i , \delta_N^j) \in
\mathcal{R}$}, \\
0 & \text{otherwise}.
\end{cases}
\end{equation*}
If $C$ is a matrix which has the same set of rows as
$A_{\mathcal{R}}$ but with no rows repeated, then it must be a
logical matrix with full row rank and fulfill condition (\ref{eq4})
\cite[Lemma 4.6]{r.li}.

\begin{example} \label{eg1} \normalfont
To illustrate this fact, as well as the main idea behind obtaining
an algebraic representation, we consider a BCN as in (\ref{eq1}),
with $n = 3$ and $m = 1$. The corresponding Boolean functions are
given by the truth table shown in Table~\ref{tab1}.
\begin{table}
\caption{Truth table for Example~\ref{eg1}.} \label{tab1} \centering
\begin{tabular}{llll||llll}
\hline $u \, x_1 \, x_2 \, x_3$ & $f_1$ & $f_2$ & $f_3$ & $u \, x_1
\, x_2 \, x_3$ & $f_1$ & $f_2$ & $f_3$ \\
\hline $1 \;\,1 \;\, 1 \;\, 1$ & $1$ & $1$ & $0$ & $0\;\, 1\;\,
1\;\, 1$ & $1$ & $1$ & $1$ \\
$1 \;\,1 \;\, 1 \;\, 0$ & $1$ & $1$ & $1$ & $0\;\, 1\;\, 1\;\, 0$ &
$1$ & $1$ & $1$ \\
$1 \;\,1 \;\, 0 \;\, 1$ & $1$ & $1$ & $1$ & $0\;\, 1\;\, 0\;\, 1$ &
$1$ & $1$ & $1$ \\
$1 \;\,1 \;\, 0 \;\, 0$ & $0$ & $1$ & $1$ & $0\;\, 1\;\, 0\;\, 0$ &
$0$ & $0$ & $0$ \\
$1 \;\,0 \;\, 1 \;\, 1$ & $0$ & $1$ & $0$ & $0\;\, 0\;\, 1\;\, 1$ &
$0$ & $1$ & $0$ \\
$1 \;\,0 \;\, 1 \;\, 0$ & $0$ & $0$ & $1$ & $0\;\, 0\;\, 1\;\, 0$ &
$0$ & $0$ & $1$ \\
$1 \;\,0 \;\, 0 \;\, 1$ & $0$ & $0$ & $0$ & $0\;\, 0\;\, 0\;\, 1$ &
$0$ & $0$ & $0$ \\
$1 \;\,0 \;\, 0 \;\, 0$ & $0$ & $1$ & $1$ & $0\;\, 0\;\, 0\;\, 0$ &
$0$ & $0$ & $1$ \\
\hline
\end{tabular}
\end{table}
Since $n = 3$ and $m = 1$, the size of the matrix $F$ in the
algebraic representation is $8 \times 16$. To find this matrix, we
see from Table~\ref{tab1} that if $u(t) = x_1(t) = x_2(t) = x_3(t) =
1$, we have $x_1(t+1) = x_2(t+1) = 1$, and $x_3(t+1) = 0$. In the
algebraic framework, this corresponds to $u(t) = x_1(t) = x_2(t) =
x_3(t) = \delta_2^1$, $x_1(t+1) =  x_2(t+1) = \delta_2^1$, and
$x_3(t+1) = \delta_2^2$, so
\begin{align*}
&x(t+1) = \delta_2^1 \ltimes \delta_2^1 \ltimes \delta_2^2 =
\delta_8^2, \\
&u(t) \ltimes x(t) = \delta_2^1 \ltimes \delta_2^1 \ltimes
\delta_2^1 \ltimes \delta_2^1 = \delta_{16}^1 .
\end{align*}
Substituting these to the left- and right-hand sides of (\ref{eq2})
yields
\begin{equation} \label{eq4b}
\delta_8^2 = F \ltimes \delta_{16}^1 = F \delta_{16}^1 .
\end{equation}
The second equality follows since the semitensor product is nothing
but the standard product if the multiplied matrices (or vectors)
have compatible sizes \cite{chengbook}. From (\ref{eq4b}), and
considering that right-multiplying a matrix by a canonical vector
yields the corresponding column of the matrix, we know that the
first column of $F$ is $\delta_8^2$. Repeating a similar argument
for each combination in the truth table, we can determine all the
columns of $F$, i.e., we determine the second column of $F$ by
considering the case when $u(t) = x_1 (t) = x_2 (t) = 1$ and $x_3(t)
= 0$, the third column by considering $u(t) = x_1(t) = 1$, $x_2(t) =
0$, and $x_3(t) = 1$, and so on. The matrix we get is
\begin{equation} \label{eq4c} \setcounter{MaxMatrixCols}{20}
F = \big[
\begin{matrix}
\delta_8^2 & \delta_8^1 & \delta_8^1 & \delta_8^5 & \delta_8^6 &
\delta_8^7 & \delta_8^8 & \delta_8^5 & \delta_8^1 &\delta_8^1 &
\delta_8^1 & \delta_8^8 & \delta_8^6 & \delta_8^7 & \delta_8^8 &
\delta_8^7
\end{matrix} \big].
\end{equation}

Consequently, the algebraic representation of this BCN is given by
\begin{equation} \label{eq4d}
x(t+1) = F \ltimes u(t) \ltimes x(t) , \;\; x(t) \in \Delta_8, \;\;
u(t) \in \Delta_2 ,
\end{equation}
with $F$ found above. Note that system (\ref{eq4d}) evolves on the
set $\Delta_8 = \{\delta_8^1, \ldots, \delta_8^8\}$, and each
canonical vector $\delta_8^i$ corresponds to a possible
configuration of the BCN (e.g., $\delta_8^1$ corresponds to
$[1,1,1]$ since $\delta_2^1 \ltimes \delta_2^1 \ltimes \delta_2^1 =
\delta_8^1$, $\delta_8^2$ corresponds to $[1,1,0]$ since $\delta_2^1
\ltimes \delta_2^1 \ltimes \delta_2^2 = \delta_8^2$, etc.). The
trajectories of (\ref{eq4d}) are shown in Fig.~\ref{fig1}.
\begin{figure} [b!] \label{fig1}
\begin{center}
\includegraphics[width= 6.5 cm]{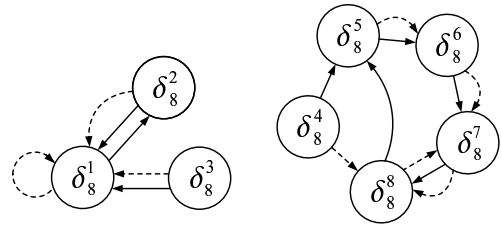}
\caption{Trajectories of system (\ref{eq4d}), which represents the
BCN in Example~\ref{eg1}. A solid line denotes the transition
corresponding to $u(t) = \delta_2^1$ and a dashed line denotes the
transition corresponding to $u(t) = \delta_2^2$.}
\end{center}
\end{figure}
Now let $\mathcal{R}$ be the equivalence relation produced by the
partition $\{ \{\delta_8^1\}, \, \{\delta_8^2, \delta_8^3\}, \,
\{\delta_8^4\}, \, \{\delta_8^5, \delta_8^6, \delta_8^7,
\delta_8^8\} \}$; that is, the pair $(a,b) \in \mathcal{R}$ if and
only if $a$ and $b$ are in the same subset of the partition. By
definition, the matrix that represents $\mathcal{R}$ has a $1$ as
its $(i,j)$-entry when $\delta_8^i$ is related to $\delta_8^j$, and
a $0$ in this position if $\delta_8^i$ is not related to
$\delta_8^j$. Accordingly, we get the following matrix for
$\mathcal{R}$:
\begin{equation*}
A_{\mathcal{R}} =
\begin{bmatrix}
1 \; & 0 \; & 0 \; & 0 \\
0 \; & J_2 \; & 0 \; & 0 \\
0 \; & 0 \; & 1 \; & 0 \\
0 \; & 0 \; & 0 \; & J_4
\end{bmatrix} ,
\end{equation*}
where $J_k$ denotes the all-one matrix of size $k \times k$.
Collapsing the identical rows of $A_{\mathcal{R}}$ yields
\begin{equation} \label{eq4a}
C = \big[
\begin{matrix}
\delta_4^1 & \delta_4^2 & \delta_4^2 & \delta_4^3 & \delta_4^4 &
\delta_4^4 & \delta_4^4 & \delta_4^4
\end{matrix} \big] .
\end{equation}
It is clear that $C$ is a full row rank logical matrix and that
(\ref{eq4}) holds.
\end{example}

\begin{remark} \label{rmk1} \normalfont
Note that if $\mathcal{R}$ is an equivalence relation on $\Delta_N$
induced by a matrix $C \in \mathcal{L}^{\widetilde{N} \times N}$ of
full row rank, then the quotient set $\Delta_N / \mathcal{R}$ is of
cardinality $\widetilde{N}$, and the correspondence $[x] \mapsto Cx$
gives a bijection between the sets $\Delta_N / \mathcal{R}$ and
$\Delta_{\widetilde{N}}$.
\end{remark}

We now consider quotients of (\ref{eq3}). We note that the BCN
(\ref{eq3}) naturally generates a transition system
$\mathcal{T}(\Sigma) = (\Delta_N, \Delta_M, \rightarrow)$, where
\begin{equation} \label{eq5}
x \overset{u}\rightarrow x' \Longleftrightarrow x' = F \ltimes u
\ltimes x .
\end{equation}
(In other words, a transition $x \overset{u}\rightarrow x'$ occurs
in $\mathcal{T}(\Sigma)$ if $u$ steers $\Sigma$ from $x$ to $x'$.)
Let $\mathcal{R}$ be an equivalence relation induced by a full row
rank logical matrix $C$ of size $\widetilde{N} \times N$. Then the
quotient transition system $\mathcal{T}(\Sigma) / \mathcal{R} =
(\Delta_N / \mathcal{R}, \Delta_M, \rightarrow_{\mathcal{R}})$ can
be thought of as having the state set $\Delta_{\widetilde{N}}$; and
the transition relation is then given by
\begin{align} \label{eq6}
z \overset{u} \rightarrow_{\mathcal{R}} z' &\Longleftrightarrow
\text{there exists a transition $x \overset{u}\rightarrow x'$ of
$\mathcal{T}(\Sigma)$} \notag \\
& \qquad \;\, \text{with $z = Cx$ and $z' = C x'$}
\end{align}
(cf. Definition~\ref{def2} and Remark~\ref{rmk1}). For the analysis
to remain in the Boolean context, we expect that the transitions of
$\mathcal{T}(\Sigma) / \mathcal{R}$ are also generated by a Boolean
system. (Here, and below, we use the term ``Boolean system" to refer
to a system of the form (\ref{eq3}) where $N$ and $M$ are not
restricted to be powers of $2$.) It is readily seen that this is the
case if and only if for any $z \in \Delta_{\widetilde{N}}$ and any
$u \in \Delta_M$, there is a unique transition $z \overset{u}
\rightarrow_{\mathcal{R}} z'$ of $\mathcal{T}(\Sigma) /
\mathcal{R}$.\footnote{Note that this is equivalent to only
requiring $\mathcal{T}(\Sigma)/\mathcal{R}$ to be
\textit{deterministic} (i.e., there do not exist transitions of the
form $z \overset{u}\rightarrow_{\mathcal{R}} z'$ and $z
\overset{u}\rightarrow_{\mathcal{R}} z''$ with $z' \neq z''$), since
for any $z \in \Delta_{\widetilde{N}}$ and $u \in \Delta_M$ there
always exists at least one $z' \in \Delta_{\widetilde{N}}$ such that
$z \overset{u}\rightarrow_{\mathcal{R}} z'$.} By (\ref{eq4}),
(\ref{eq5}) and (\ref{eq6}), the latter is equivalent to the
requirement that
\begin{align} \label{eq7}
(a,b) \in \mathcal{R} &\Longleftrightarrow (F \ltimes u \ltimes a ,
F \ltimes u \ltimes b) \in \mathcal{R} \notag \\
& \qquad \;\;\, \text{for all $u \in \Delta_M$}.
\end{align}
We therefore restrict our attention to those $\mathcal{R}$
satisfying (\ref{eq7}).

\begin{remark} \label{rmk2} \normalfont
The meaning of condition (\ref{eq7}) is clear: if we think of
$\mathcal{R}$ as a partition of $\Delta_N$, then the successor set
of each block in this partition is included in a single block of the
partition.
\end{remark}

The following theorem gives a method for explicitly constructing a
Boolean system that generates the transitions of
$\mathcal{T}(\Sigma) / \mathcal{R}$.

\begin{theorem} \label{thm1} \itshape
Consider a BCN $\Sigma$ as in (\ref{eq3}). Suppose that
$\mathcal{R}$ is an equivalence relation on $\Delta_N$ induced by a
matrix $C \in \mathcal{L}^{\widetilde{N} \times N}$ of full row
rank, and that property (\ref{eq7}) holds. For each $1 \leq k \leq
M$, let $F_k$ be the matrix in $\mathcal{L}^{N \times N}$ defined by
$F_k = F \ltimes \delta_M^k$, and let $\widetilde{F}_k = C \odot F_k
\odot C^\top$. Then:
\begin{enumerate} [(a)]
\item $\widetilde{F}_k \in \mathcal{L}^{\widetilde{N} \times
\widetilde{N}}$ for $1 \leq k \leq M$.

\item Let
\begin{equation*} \hspace{-1.5em}
\Sigma_{\mathcal{R}} \colon \;\, x_{\mathcal{R}}(t+1) =
\widetilde{F} \ltimes u(t) \ltimes x_{\mathcal{R}}(t) , \;\;
x_{\mathcal{R}}\in \Delta_{\widetilde{N}}, \;\; u \in \Delta_M
\end{equation*}
be the system where $\widetilde{F} = \big[
\begin{matrix}
\widetilde{F}_1 & \widetilde{F}_2 & \cdots & \widetilde{F}_M
\end{matrix} \big]$. If an input $u \in \Delta_M$ steers $\Sigma$
from a state $a \in \Delta_N$ to a state $a' \in \Delta_N$, then it
also steers $\Sigma_{\mathcal{R}}$ from $Ca$ to $C a'$. Conversely,
if $u$ steers $\Sigma_{\mathcal{R}}$ from a state $q \in
\Delta_{\widetilde{N}}$ to a state $q' \in \Delta_{\widetilde{N}}$,
then there is a one-step transition of $\Sigma$ from some $a \in
\Delta_N$ to some $a' \in \Delta_N$ with $Ca = q$ and $Ca' = q'$,
under this input $u$.
\end{enumerate}
\end{theorem}

\begin{proof}
(a) It is clear that each $\widetilde{F}_k$ is a $(0,1)$-matrix of
size $\widetilde{N} \times \widetilde{N}$. So we need only show
that, for $1 \leq k \leq M$, every column of $\widetilde{F}_k$
contains exactly one $1$. Let $1 \leq k \leq M$ and $1 \leq j \leq
\widetilde{N}$ be fixed. Since $C$ (being logical) has no zero rows,
there exists $1 \leq s \leq N$ such that $(C)_{js} = 1$. Choose $1
\leq r \leq N$ so that $\delta_N^r = F \ltimes \delta_M^k \ltimes
\delta_N^s$. Then $(F_k)_{rs} = 1$. For this $r$, let $1 \leq i \leq
\widetilde{N}$ be such that $(C)_{ir} = 1$. Then, by the definition
of Boolean matrix multiplication, the $(i,j)$-entry of
$\widetilde{F}_k$ is equal to $\bigvee_{p=1}^N \bigvee_{l=1}^N
[(C)_{ip} \wedge (F_k)_{pl} \wedge (C)_{jl}]$, and hence equal to
$1$ (since $(C)_{ir} = (F_k)_{rs} = (C)_{js} = 1$). This means that
each column of $\widetilde{F}_k$ has at least one $1$. Now suppose
that there is another $i'$ with $1 \leq i' \leq \widetilde{N}$ such
that $(\widetilde{F}_k)_{i'j} = 1$. Then we must have $(C)_{i'r'} =
1$, $(F_k)_{r's'} = 1$, and $(C)_{js'} = 1$ for some $1 \leq r', s'
\leq N$. These imply that $C \delta_N^{r'} =
\delta_{\widetilde{N}}^{i'}$, $C \delta_N^{s'} =
\delta_{\widetilde{N}}^j$, and $\delta_N^{r'} = F \ltimes \delta_M^k
\ltimes \delta_N^{s'}$. Since $(C)_{js} = 1$, we have $C \delta_N^s
= \delta_{\widetilde{N}}^j$ and, thus, $(\delta_N^s, \delta_N^{s'})
\in \mathcal{R}$. By (\ref{eq7}), it follows that $(F \ltimes
\delta_M^k \ltimes \delta_N^s, F \ltimes \delta_M^k \ltimes
\delta_N^{s'}) \in \mathcal{R}$, that is, $(\delta_N^r,
\delta_N^{r'}) \in \mathcal{R}$. Hence, $\delta_{\widetilde{N}}^i =
C \delta_N^r = C \delta_N^{r'} = \delta_{\widetilde{N}}^{i'}$, which
shows that $i = i'$. Thus, there is a unique $1$ in each column of
$F_k$.

(b) We first note that the system $\Sigma_{\mathcal{R}}$ is well
defined since, by (a), $\widetilde{F}$ is a logical matrix of size
$\widetilde{N} \times \widetilde{N}M$. Let $1 \leq r,s \leq N$, let
$1 \leq k \leq M$, and assume that the input $u = \delta_M^k$ steers
$\Sigma$ from $\delta_N^s$ to $\delta_N^r$. We have $(F_k)_{rs} =
1$. Suppose that $C \delta_N^s = \delta_{\widetilde{N}}^j$ and $C
\delta_N^r = \delta_{\widetilde{N}}^i$. Then $(C)_{js} = (C)_{ir} =
1$ and, hence, $(\widetilde{F}_k)_{ij} = 1$ by the definition of the
Boolean product. This combined with (a) implies that
$\delta_{\widetilde{N}}^i = \widetilde{F} \ltimes \delta_M^k \ltimes
\delta_{\widetilde{N}}^j$; in other words, the input $u =
\delta_M^k$ steers $\Sigma_{\mathcal{R}}$ from
$\delta_{\widetilde{N}}^j$ to $\delta_{\widetilde{N}}^i$.

Conversely, let $1 \leq i,j \leq \widetilde{N}$ and suppose that the
input $u = \delta_M^k$ takes $\Sigma_{\mathcal{R}}$ from
$\delta_{\widetilde{N}}^j$ to $\delta_{\widetilde{N}}^i$. Then
$(\widetilde{F}_k)_{ij} = 1$, and hence there must be some $1 \leq
r,s \leq N$ such that $(C)_{ir} = 1$, $(F_k)_{rs} = 1$, and
$(C)_{js} = 1$. Thus, $C \delta_N^s = \delta_{\widetilde{N}}^j$, $C
\delta_N^r = \delta_{\widetilde{N}}^i$, and $\Sigma$ can be driven
from $\delta_N^s$ to $\delta_N^r$ with the input $u = \delta_M^k$.
\end{proof}

Since, by the above theorem, $\Sigma_{\mathcal{R}}$ generates the
transitions of $\mathcal{T}(\Sigma) / \mathcal{R}$ (cf.
(\ref{eq6})), it can be interpreted as a quotient of the BCN
$\Sigma$.

\begin{example} \label{eg2} \normalfont
Consider the BCN in Example~\ref{eg1}. The matrix $F$ in the
algebraic representation is given by (\ref{eq4c}). Let $C$ be as in
(\ref{eq4a}) and let $\mathcal{R}$ be the equivalence relation
defined in Example~\ref{eg1}, induced by $C$. It is easy to check
that $\mathcal{R}$ satisfies (\ref{eq7}). Set $F_1 = F \ltimes
\delta_2^1$ and $F_2 = F \ltimes \delta_2^2$. A calculation yields
\begin{align*}
\widetilde{F}_1 &= C \odot F_1 \odot C^\top = \big[
\begin{matrix}
\delta_4^2 & \delta_4^1 & \delta_4^4 & \delta_4^4
\end{matrix} \big] , \\
\widetilde{F}_2 &= C \odot F_2 \odot C^\top = \big[
\begin{matrix}
\delta_4^1 & \delta_4^1 & \delta_4^4 & \delta_4^4
\end{matrix} \big] .
\end{align*}
Fig.~\ref{fig2} shows the trajectories of $\Sigma_{\mathcal{R}}$
with $\widetilde{N} = 4$, $M = 2$, and $\widetilde{F} \in
\mathcal{L}^{4 \times 8}$ given by $\widetilde{F} = \big[
\begin{matrix}
\widetilde{F}_1 & \widetilde{F}_2
\end{matrix} \big]$.
\begin{figure} \label{fig2}
\begin{center}
\includegraphics[width= 7.3 cm]{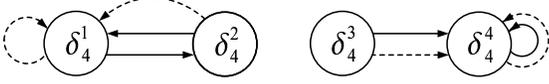}
\caption{Trajectories of the Boolean system $\Sigma_{\mathcal{R}}$
defined in Example~\ref{eg2}. A solid (resp. dashed) line represents
the transition resulting from $u(t) = \delta_2^1$ (resp. $u(t) =
\delta_2^2$).}
\end{center}
\end{figure}
We see from the figure that $\Sigma_{\mathcal{R}}$ is indeed a
quotient of the original BCN, which does not distinguish between
states related by $\mathcal{R}$.
\end{example}

Using Theorem~\ref{thm1} one can obtain a quotient Boolean system,
once an equivalence relation satisfying property (\ref{eq7}) is
found. In the next subsection, we will address the issue of
computing equivalence relations which allow the construction of
quotient Boolean systems.

\subsection{Computing equivalence relations}

\label{sec3.2}

Precisely, in this subsection we are concerned with the following
problem: given an equivalence relation $\mathcal{S}$ on $\Delta_N$,
determine the \textit{maximal} (with respect to set inclusion)
equivalence relation $\mathcal{R}$ on $\Delta_N$ such that
$\mathcal{R} \subseteq \mathcal{S}$ and (\ref{eq7}) holds. Here the
relation $\mathcal{S}$ may be interpreted as a preliminary
classification of the states of a BCN; see Section~\ref{sec4} below
for specific instances. We are interested in finding the maximal
equivalence relation since in many cases we want the size of the
quotient system to be as small as possible.

First, we remark that such a maximal equivalence relation always
exists and it is unique, as shown in the following proposition.

\begin{proposition} \label{prop1} \itshape
Let $\mathcal{S}$ be an equivalence relation on $\Delta_N$. Then the
set of all relations $\mathcal{R} \subseteq \Delta_N \times
\Delta_N$ that are contained in $\mathcal{S}$ and satisfy property
(\ref{eq7}) has a unique maximal element (with respect to set
inclusion), and the maximal element is an equivalence relation on
$\Delta_N$.
\end{proposition}

\begin{proof}
Note that the identity relation $\mathcal{R}_{\text{id}} = \{(a,a)
\colon a \in \Delta_N \}$ satisfies (\ref{eq7}) and
$\mathcal{R}_{\text{id}} \subseteq \mathcal{S}$ (since $\mathcal{S}$
is reflexive). Also note that if two relations $\mathcal{R}_1
\subseteq \mathcal{S}$ and $\mathcal{R}_2 \subseteq \mathcal{S}$
both satisfy property (\ref{eq7}), then the same is true for their
union $\mathcal{R}_1 \cup \mathcal{R}_2$. The first statement
follows immediately.

The maximal element $\widetilde{\mathcal{R}}$ is reflexive since it
contains the identity relation $\mathcal{R}_{\text{id}}$. To show
the symmetry and transitivity of $\widetilde{\mathcal{R}}$, consider
the inverse relation $\widetilde{\mathcal{R}}^{-1} = \{(b,a) \colon
(a,b) \in \widetilde{\mathcal{R}}\}$ and the composition
$\widetilde{\mathcal{R}} \circ \widetilde{\mathcal{R}} = \{(a,c)
\colon \text{there exists $b \in \Delta_N$ such that $(a,b) \in
\widetilde{\mathcal{R}}$ and $(b,$}$ $c) \in \widetilde{\mathcal{R}}
\}$. It is easy to see that both $\widetilde{\mathcal{R}}^{-1}$ and
$\widetilde{\mathcal{R}} \circ \widetilde{\mathcal{R}}$ satisfy
(\ref{eq7}), and are contained in $\mathcal{S}$ since
$\widetilde{\mathcal{R}} \subseteq \mathcal{S}$ and $\mathcal{S}$ is
symmetric and transitive. Hence, $\widetilde{\mathcal{R}}$ contains
$\widetilde{\mathcal{R}}^{-1}$ and $\widetilde{\mathcal{R}} \circ
\widetilde{\mathcal{R}}$, implying that $\widetilde{\mathcal{R}}$ is
symmetric and transitive. The second statement is proved.
\end{proof}

The following theorem suggests a way of computing such an
equivalence relation.

\begin{theorem} \label{thm2} \itshape
Let $F \in \mathcal{L}^{N \times NM}$, and let $\mathcal{S}$ be an
equivalence relation on $\Delta_N$. For each $u \in \Delta_M$ define
a relation $\mathcal{S}_u$ on $\Delta_N$ by: $(a,a') \in
\mathcal{S}_u$ if and only if $a' = F \ltimes u \ltimes a$. Define a
sequence of relations $\mathcal{R}_k$ by
\begin{equation*}
\mathcal{R}_1 = \mathcal{S} \;\; \text{and} \;\; \mathcal{R}_{k+1} =
\bigg(\bigcap_{u \in \Delta_M} (\mathcal{S}_u^{-1} \circ
\mathcal{R}_k \circ \mathcal{S}_u)\bigg) \cap \mathcal{R}_k .
\end{equation*}
Then:
\begin{enumerate} [(a)]
\item The sequence of relations $\mathcal{R}_1, \mathcal{R}_2,
\ldots, \mathcal{R}_k , \ldots$ satisfies $\mathcal{R}_1 \supseteq
\mathcal{R}_2 \supseteq \cdots \supseteq \mathcal{R}_k \supseteq
\cdots$.

\item There is an integer $k^\ast$ such that $\mathcal{R}_{k^\ast +
1} = \mathcal{R}_{k^\ast}$.

\item $\mathcal{R}_{k^\ast}$ is the maximal equivalence relation on
$\Delta_N$ such that $\mathcal{R}_{k^\ast} \subseteq \mathcal{S}$
and property (\ref{eq7}) holds.
\end{enumerate}
\end{theorem}

\begin{proof}
Part (a) is quite trivial. Part (b) follows from (a) and the
finiteness of each $\mathcal{R}_k$.

We turn to the proof of (c). By Proposition~\ref{prop1}, it suffices
to show that $\mathcal{R}_{k^\ast} \subseteq \Delta_N \times
\Delta_N$ is the maximal relation satisfying $\mathcal{R}_{k^\ast}
\subseteq \mathcal{S}$ and condition (\ref{eq7}). The relation
$\mathcal{R}_{k^\ast}$ is clearly a subset of $\mathcal{S}$. To show
that (\ref{eq7}) holds true, suppose that $(a,b) \in
\mathcal{R}_{k^\ast}$ and $u \in \Delta_M$. Since
$\mathcal{R}_{k^\ast} = \mathcal{R}_{k^\ast + 1} \subseteq
\mathcal{S}_u^{-1} \circ \mathcal{R}_{k^\ast} \circ \mathcal{S}_u$,
there exist $a',b' \in \Delta_N$ such that $(a,a') \in
\mathcal{S}_u$, $(b', b) \in \mathcal{S}_u^{-1}$, and $(a',b') \in
\mathcal{R}_{k^\ast}$. It follows from the definition of
$\mathcal{S}_u$ that $a' = F \ltimes u \ltimes a$ and $b' = F
\ltimes u \ltimes b$. Hence, $(F \ltimes u \ltimes a, F \ltimes u
\ltimes b) \in \mathcal{R}_{k^\ast}$.

To prove the maximality of $\mathcal{R}_{k^\ast}$, let $\mathcal{R}
\subseteq \Delta_N \times \Delta_N$ be another relation which is
contained in $\mathcal{S}$ and satisfies (\ref{eq7}). We claim that
$\mathcal{R} \subseteq \mathcal{R}_k$ for all $k$. The case $k =
k^\ast$ completes the proof. We shall use induction on $k$. The case
$k = 1$ is trivial, so we take $k > 1$ and assume that $\mathcal{R}
\subseteq \mathcal{R}_{k-1}$. Let $(a,b) \in \mathcal{R}$. Then for
any $u \in \Delta_M$, we have $(F \ltimes u \ltimes a, F \ltimes u
\ltimes b) \in \mathcal{R} \subseteq \mathcal{R}_{k-1}$. By the
definition of $\mathcal{S}_u$, it follows that $(a, F \ltimes u
\ltimes a) \in \mathcal{S}_u$ and $(F \ltimes u \ltimes b, b) \in
\mathcal{S}_u^{-1}$. Hence, $(a,b) \in \mathcal{S}_u^{-1} \circ
\mathcal{R}_{k-1} \circ \mathcal{S}_u$, and consequently $(a,b) \in
\mathcal{R}_k$ since $u$ was arbitrary. This shows that $\mathcal{R}
\subseteq \mathcal{R}_k$, and our claim follows.
\end{proof}

For applications, it is convenient to reformulate Theorem~\ref{thm2}
in terms of $(0,1)$-matrices. Recall that a relation $\mathcal{R}$
on $\Delta_N$ can be represented by an $N \times N$ matrix, whose
$(i,j)$-entry is $1$ if $(\delta_N^i, \delta_N^j) \in \mathcal{R}$
and $0$ otherwise. So if $A_{\mathcal{R}}$ is the matrix
representing $\mathcal{R}$, then the inverse relation
$\mathcal{R}^{-1}$ has $A_{\mathcal{R}}^\top$ as the matrix
representation. Moreover, if $\mathcal{R}'$ is another relation on
$\Delta_N$ represented by $A_{\mathcal{R}'}$, then the matrices
representing $\mathcal{R} \cap \mathcal{R}'$ and $\mathcal{R}' \circ
\mathcal{R}$ are $A_{\mathcal{R}} \wedge A_{\mathcal{R}'}$ and
$A_{\mathcal{R}} \odot A_{\mathcal{R}'}$ (see, e.g., \cite[Section
9.3]{rosenbook}). Note that if $\mathcal{S}_u$ is the relation
defined in Theorem~\ref{thm2} and if $u = \delta_M^k$, then
\begin{equation*}
(\delta_N^i, \delta_N^j) \in \mathcal{S}_u \Longleftrightarrow F_k
\delta_N^i = \delta_N^j \Longleftrightarrow (F_k)_{ji} = 1 ,
\end{equation*}
where $F_k = F \ltimes \delta_M^k$, and thus $F_k^\top$ is the
matrix representing $\mathcal{S}_u$. From these facts and
Theorem~\ref{thm2}, the following corollary follows immediately.

\begin{corollary} \label{cor1} \itshape
Suppose that $\mathcal{S}$ is an equivalence relation on $\Delta_N$
represented by a matrix $A_{\mathcal{S}}$, and suppose that $F \in
\mathcal{L}^{N \times NM}$. For each $1 \leq i \leq M$, let $F_i$ be
the matrix $F_i = F \ltimes \delta_M^i$. Define a sequence of
$(0,1)$-matrices by
\begin{multline*}
A_1 = A_{\mathcal{S}} \;\, \text{and} \;\, A_{k+1} = A_k \wedge
(F_1^\top \odot A_k \odot F_1) \wedge \cdots \\
\wedge (F_M^\top \odot A_k \odot F_M) .
\end{multline*}
Then there is an integer $k^\ast$ such that $A_{k^\ast + 1} =
A_{k^\ast}$, and $A_{k^\ast}$ is the matrix representing the maximal
equivalence relation on $\Delta_N$ that is contained in
$\mathcal{S}$ and satisfies property (\ref{eq7}).
\end{corollary}

\begin{example} \label{eg3} \normalfont
Consider again the BCN in Example~\ref{eg1}. If we let $\mathcal{S}$
be the equivalence relation induced by the partition $\{
\{\delta_8^1\}, \, \{\delta_8^2, \delta_8^3, \delta_8^4\}, \,
\{\delta_8^5, \delta_8^6, \delta_8^7, \delta_8^8\} \}$, then
\begin{equation*}
A_1 =
\begin{bmatrix}
1 \; & 0 \; & 0 \\
0 \; & J_3 \; & 0 \\
0 \; & 0 \; & J_4
\end{bmatrix} ,
\end{equation*}
and a short computation yields
\begin{equation*}
A_2 = A_3 =
\begin{bmatrix}
1 \; & 0 \; & 0 \; & 0 \\
0 \; & J_2 \; & 0 \; & 0 \\
0 \; & 0 \; & 1 \; & 0 \\
0 \; & 0 \; & 0 \; & J_4
\end{bmatrix} ,
\end{equation*}
which is exactly the matrix representing the relation
given in Example~\ref{eg1}. So the relation $\mathcal{R}$ presented
in Example~\ref{eg1} is the maximal equivalence relation contained
in $\mathcal{S}$ and satisfying property (\ref{eq7}).
\end{example}

\section{Control design via quotients}

\label{sec4}

This section discusses the application of quotient systems for
control design. We consider two typical control problems in BCNs and
show how these problems can be solved through the use of a quotient
Boolean system.

\subsection{Stabilization}

\label{sec4.1}

Consider a BCN $\Sigma$ as given in (\ref{eq3}). Let $\mathcal{M}
\subseteq \Delta_N$ be a target set of states. We say that $\Sigma$
is \textit{stabilizable} to $\mathcal{M}$ if for every $x(0) \in
\Delta_N$ there exists a control sequence $\{u(0), u(1), u(2),
\ldots \}$, with $u(i) \in \Delta_M$, and a positive integer $\tau$
such that $x(t) \in \mathcal{M}$ for all $t \geq \tau$ (see, e.g.,
\cite{chengbook}). The following result shows that, by defining the
equivalence relation appropriately, we can easily obtain a
stabilizing controller for $\Sigma$ on the basis of a stabilizer for
its quotient system.

\begin{proposition} \label{prop2} \itshape
Consider a BCN $\Sigma$ as given in (\ref{eq3}). Let $\mathcal{M}
\subseteq \Delta_N$ and let $\mathcal{S}$ be the equivalence
relation on $\Delta_N$ determined by the partition $\{\,
\mathcal{M}, \: \Delta_N - \mathcal{M} \,\}$. Suppose that
$\mathcal{R}$ is an equivalence relation on $\Delta_N$ induced by a
matrix $C \in \mathcal{L}^{\widetilde{N} \times N}$ of full row
rank, $\mathcal{R} \subseteq \mathcal{S}$, and condition (\ref{eq7})
holds. Suppose $\Sigma_{\mathcal{R}}$ is defined as in
Theorem~\ref{thm1}. If $\Sigma_{\mathcal{R}}$ can be stabilized to
the set $\mathcal{M}_{\mathcal{R}} = \{ \, Cx \colon x \in
\mathcal{M} \, \}$ via a feedback law $(x_{\mathcal{R}}, t) \mapsto
u(x_\mathcal{R}, t)$, then $\Sigma$ can be stabilized to
$\mathcal{M}$ using the feedback law $(x, t) \mapsto u(Cx, t)$.
\end{proposition}

\begin{proof}
We first show that if the initial states of $\Sigma$ and
$\Sigma_{\mathcal{R}}$ satisfy $C x(0) = x_{\mathcal{R}}(0)$, then
the feedback laws $(x,t) \mapsto u(Cx, t)$ and $(x_{\mathcal{R}}, t)
\mapsto u(x_{\mathcal{R}}, t)$ generate the same input sequence, and
the trajectories satisfy $Cx(t) = x_{\mathcal{R}}(t)$ for all $t =
0,1,2, \ldots$. It is clearly true that the two feedback laws
generate the same input, say $u_0$, at $t = 0$. By the second part
of Theorem~\ref{thm1}(b), there is a one-step transition of $\Sigma$
from some $a \in \Delta_N$ to some $a' \in \Delta_N$ with $Ca =
x_{\mathcal{R}}(0)$ and $C a' = x_{\mathcal{R}}(1)$, under this
input $u_0$. Since $Cx(0) = Ca$, it follows from (\ref{eq4}) that
$(x(0), a) \in \mathcal{R}$, and then by (\ref{eq7}) we have $(x(1),
a') \in \mathcal{R}$. Thus $C x(1) = Ca' = x_{\mathcal{R}}(1)$ again
by (\ref{eq4}). The fact we want now follows by induction.

Now we can prove the proposition. Assume that the feedback laws
$(x,t) \mapsto u(Cx, t)$ and $(x_{\mathcal{R}}, t) \mapsto
u(x_{\mathcal{R}}, t)$ are applied to $\Sigma$ and
$\Sigma_{\mathcal{R}}$, respectively. Let $p \in \Delta_N$ and let
$q = C p$. Then there is a $\tau$ such that the trajectory of
$\Sigma_{\mathcal{R}}$ with $x_{\mathcal{R}}(0) = q$ satisfies
$x_{\mathcal{R}}(t) \in \mathcal{M}_{\mathcal{R}}$ for all $t \geq
\tau$. Since the trajectory of $\Sigma$ with $x(0) = p$ always
satisfies $C x(t) = x_{\mathcal{R}}(t)$, to each $t \geq \tau$ there
corresponds some $b \in \mathcal{M}$ such that $C x(t) = C b$, and
hence $(x(t), b) \in \mathcal{R} \subseteq \mathcal{S}$. This forces
$x(t) \in \mathcal{M}$ whenever $t \geq \tau$, since $\mathcal{S}$
is the equivalence relation yielded by the partition $\{\,
\mathcal{M}, \: \Delta_N - \mathcal{M} \,\}$. Since $p$ was
arbitrary, we conclude that stabilization of $\Sigma$ to
$\mathcal{M}$ is achieved, via the feedback law $(x,t) \mapsto u(Cx,
t)$.
\end{proof}

\begin{remark} \label{rmk3} \normalfont
Note that in Proposition~\ref{prop2} we do not assume $\mathcal{R}$
to be maximal, although that will be the case in most applications
of the proposition. A similar remark applies to
Proposition~\ref{prop3} below.
\end{remark}

\subsection{Optimal control}

\label{sec4.2}

As another example of application we consider the following
finite-horizon optimal control problem, introduced in
\cite{fornasini2014}.

\begin{problem} \label{plm1} \normalfont
Consider a BCN $\Sigma$ as in (\ref{eq3}). Given an initial state
$x_0$ and a finite time horizon $T \in \mathbb{Z}^+$, find a control
sequence that minimizes the cost function
\begin{equation} \label{eq8a}
J = \sum_{t=0}^{T-1} l(u(t), x(t)) + g(x(T)) ,
\end{equation}
where $l(u, x)$ and $g(x)$ are functions defined on $\Delta_M \times
\Delta_N$ and $\Delta_N$, respectively.
\end{problem}

We show that the solution to Problem~\ref{plm1} for $\Sigma$ can be
easily derived on the basis of a solution to Problem~\ref{plm1} for
a suitably chosen quotient system. Let $\mathcal{S}$ be the
equivalence relation on $\Delta_N$ given by
\begin{align} \label{eq8}
(x, x') \in \mathcal{S} &\Longleftrightarrow \text{$g(x) = g(x')$
and} \notag \\
& \qquad \;\, \text{$l(u, x) = l(u, x')$ for all $u \in \Delta_M$}.
\end{align}
We observe that, for a matrix $C \in \mathcal{L}^{\widetilde{N}
\times N}$ with full row rank, if the equivalence relation
$\mathcal{R}$ induced by $C$ satisfies $\mathcal{R} \subseteq
\mathcal{S}$, then the following two maps are well defined:
\begin{align}
&l_{\mathcal{R}} \colon \Delta_M \times \Delta_{\widetilde{N}}
\rightarrow \mathbb{R}, \;\; (u, a) \mapsto l(u, x) \notag \\
&\qquad \qquad \qquad \qquad \qquad \quad \text{whenever $a = Cx$},
\label{eq9} \\
&g_{\mathcal{R}} \colon \Delta_{\widetilde{N}}
\rightarrow \mathbb{R}, \;\; \text{$a \mapsto g(x)$ whenever $a =
Cx$}. \label{eq10}
\end{align}
Based on this observation, we can state the following proposition.

\begin{proposition} \label{prop3} \itshape
Let $\Sigma$ be a BCN described by (\ref{eq3}). Suppose that
$\mathcal{S}$ is the equivalence relation on $\Delta_N$ given by
(\ref{eq8}), $\mathcal{R}$ is an equivalence relation on $\Delta_N$
induced by a full row rank logical matrix $C \in
\mathcal{L}^{\widetilde{N} \times N}$, $\mathcal{R} \subseteq
\mathcal{S}$, and (\ref{eq7}) holds. Consider Problem~\ref{plm1}
with given $x_0$, $T$, and $J$. Let $\Sigma_{\mathcal{R}}$ be the
Boolean system constructed in Theorem~\ref{thm1}, and define
$J_{\mathcal{R}} = \sum_{t=0}^{T-1} l_{\mathcal{R}}(u(t),
x_{\mathcal{R}}(t)) + g_{\mathcal{R}}(x_{\mathcal{R}}(T))$, where
$l_{\mathcal{R}}$ and $g_{\mathcal{R}}$ are given by (\ref{eq9}) and
(\ref{eq10}).
\begin{enumerate} [(a)]
\item If $U^\ast = \{u^\ast(0), \ldots, u^\ast(T-1)\}$ is an optimal
control sequence solving Problem~\ref{plm1} with $\Sigma$, $x_0$,
and $J$ replaced by $\Sigma_{\mathcal{R}}$, $x_{\mathcal{R}}^0 = C
x_0$, and $J_{\mathcal{R}}$, respectively, then $U^\ast$ is also an
optimal control for $\Sigma$. Moreover, let $J^\ast$ be the optimal
cost $\min_{u(\cdot)}J$ under the initial condition $x(0) = x_0$ and
let $J_{\mathcal{R}}^\ast$ be the optimal cost $\min_{u(\cdot)}
J_{\mathcal{R}}$ under the condition $x_{\mathcal{R}}(0) = C x_0$.
Then $J^\ast = J_{\mathcal{R}}^\ast$.

\item If $(x_{\mathcal{R}}, t) \mapsto u^\ast (x_{\mathcal{R}}, t)$
is an optimal control policy\footnote{It was shown in
\cite{fornasini2014} that the optimal control input can always be
implemented as a time-varying feedback from the states.} solving
Problem~\ref{plm1} with $\Sigma$ and $J$ replaced by
$\Sigma_{\mathcal{R}}$ and $J_{\mathcal{R}}$, respectively, then the
control policy given by $(x, t) \mapsto u^\ast (Cx, t)$ is an
optimal control policy for $\Sigma$.
\end{enumerate}
\end{proposition}

\begin{proof}
(a) An argument similar to the first paragraph of the proof of
Proposition~\ref{prop2} shows that, if the initial states of
$\Sigma$ and $\Sigma_{\mathcal{R}}$ satisfy $C x(0) =
x_{\mathcal{R}}(0)$, then for any control sequence $u(0), \ldots,
u(T-1)$, the corresponding trajectories satisfy $C x(t) =
x_{\mathcal{R}}(t)$ for $t = 0, \ldots, T$, and hence $g(x(T)) =
g_{\mathcal{R}}(x_{\mathcal{R}}(T))$ and $l(u(t), x(t)) =
l_{\mathcal{R}} (u(t), x_{\mathcal{R}}(t))$ for each $0 \leq t \leq
T-1$, so that the cost functions $J$ and $J_{\mathcal{R}}$ return
the same value. This implies that if $U^\ast$ minimizes
$J_{\mathcal{R}}$ with the initial condition $x_{\mathcal{R}}(0) = C
x_0$, then it also minimizes $J$ subject to $x(0) = x_0$, and
moreover, the associated optimal costs $J^\ast$ and
$J_{\mathcal{R}}^\ast$ are equal.

Part (b) follows directly from (a) and the fact (explained in the
first paragraph of the proof of Proposition~\ref{prop2}) that the
feedback laws $(x,t) \mapsto u^\ast (Cx, t)$ and $(x_{\mathcal{R}},
t) \mapsto u^\ast (x_{\mathcal{R}}, t)$ generate the same control
sequence whenever $C x(0) = x_{\mathcal{R}}(0)$.
\end{proof}

\begin{example} \label{eg4} \normalfont
To give an intuitive example of the equivalence relation defined by
(\ref{eq8}), suppose that $M = 2$, $N = 4$, and the functions $l
\colon \Delta_2 \times \Delta_4 \rightarrow \mathbb{R}$ and $g
\colon \Delta_4 \rightarrow \mathbb{R}$ are given by
\begin{align*}
&l(\delta_2^1, \delta_4^1) = 1, \quad l(\delta_2^1, \delta_4^2) =
l(\delta_2^1, \delta_4^3) = l(\delta_2^1, \delta_4^4) = 2 , \\
&l(\delta_2^2, x) = 3 \quad (x \in \Delta_4) , \\
&g(\delta_4^1) = g(\delta_4^2) = g(\delta_4^3) = 1 , \quad
g(\delta_4^4) = 2 .
\end{align*}
First, by definition the relation $\mathcal{S}$ contains all pairs
of the form $(a,a)$, namely, $(\delta_4^1, \delta_4^1)$,
$(\delta_4^2, \delta_4^2)$, $(\delta_4^3, \delta_4^3)$, and
$(\delta_4^4, \delta_4^4)$. Second, note that $g(\delta_4^2) =
g(\delta_4^3) = 1$, $l(\delta_2^1, \delta_4^2) = l(\delta_2^1,
\delta_4^3) = 2$, and $l(\delta_2^2, \delta_4^2) = l(\delta_2^2,
\delta_4^3) = 3$. Thus, both pairs $(\delta_4^2, \delta_4^3)$ and
$(\delta_4^3, \delta_4^2)$ belong to $\mathcal{S}$. Moreover, it is
easily checked that they are the only pairs of distinct states that
satisfy $g(x) = g(x')$, $l(\delta_2^1, x) = l(\delta_2^1, x')$, and
$l(\delta_2^2, x) = l(\delta_2^2, x')$ simultaneously. Hence no pair
other than those listed belongs to $\mathcal{S}$.
\end{example}

\begin{remark} \label{rmk4} \normalfont
It is noted in \cite{fornasini2014} that the cost function described
in (\ref{eq8a}) can be equivalently expressed in a linear form as $J
= \sum_{t=0}^{T-1} \theta \ltimes u(t) \ltimes x(t) + \mu x(T)$,
where $\mu$ is a row vector of $N$ components and $\theta =
[\theta_1, \, \theta_2, \, \ldots, \, \theta_M]$ with each
$\theta_i$ being an $N$-component row vector. We remark that the
index $J_{\mathcal{R}}$ appearing in Proposition~\ref{prop3} is
easily obtained from this expression. In fact, since the function
$g_{\mathcal{R}}$ is defined on $\Delta_{\widetilde{N}}$, it can be
expressed in the form $g_{\mathcal{R}}(x) = \mu_{\mathcal{R}} x$ for
some $\widetilde{N}$-component row vector $\mu_{\mathcal{R}}$. Let
$C$ be as in Proposition~\ref{prop3}. Then by (\ref{eq10}) we have
$\mu_{\mathcal{R}} C = \mu$ and so $\mu_{\mathcal{R}} = \mu C^+$,
where $C^+ = C^\top (C C^\top)^{-1}$ is the pseudoinverse of $C$. In
a similar manner, the function $l_{\mathcal{R}}$ defined by
(\ref{eq9}) can be equivalently expressed as $l_{\mathcal{R}}(u, x)
= \theta_{\mathcal{R}} \ltimes u \ltimes x$, where
$\theta_{\mathcal{R}} = [\theta_1', \, \theta_2', \, \ldots, \,
\theta_M']$ with $\theta_i' = \theta_i C^+$ for each $i$. Thus the
index cost $J_{\mathcal{R}}$ can be rewritten in a linear form as
follows: $J_{\mathcal{R}} = \sum_{t=0}^{T-1} \theta_{\mathcal{R}}
\ltimes u(t) \ltimes x_{\mathcal{R}}(t) + \mu_{\mathcal{R}}
x_{\mathcal{R}}(T)$.
\end{remark}

One can obtain analogs of Proposition~\ref{prop3} for other kinds of
optimal control problems (such as the infinite-horizon optimal or
average-cost optimal problems \cite{fornasini2014}). The essence of
the arguments is the same as that of Proposition~\ref{prop3}, and so
we omit them.

\subsection{Comparative simulations}

\label{sec4.3}

The proposed methods have been tested on several randomly generated
$16$-node networks. Recall that a BCN expressed by (\ref{eq1})
consists of two types of nodes, namely, internal nodes ($x_1,
\ldots, x_n$) and external control nodes ($u_1, \ldots, u_m$). We
considered the cases of $m = 1$, $2$, $3$, and $5$. When $m = 1$,
there are $15$ internal nodes and $1$ control nodes; the original
network size is $2^{15} = 32768$. When $m = 2$, there are $14$
internal nodes and $2$ control nodes; the original network size is
$2^{14} = 16384$. When $m = 3$, there are $13$ internal nodes, so
the original network size is $2^{13} = 8192$, and when $m = 5$ the
original network size is $2^{11} = 2048$. First, we evaluate the
efficiency of the quotient-based method given in
Proposition~\ref{prop2}. The target sets $\mathcal{M}$ of the
stabilization problem were randomly selected, with cardinality $k =
1$ and $k = 100$. Table~\ref{tab2} shows the numerical results
obtained for different combinations of $m$ and $k$.
\begin{table}
\caption{Comparison between controller design done with the
quotient-based method and done the conventional way.} \label{tab2}
\centering
\begin{tabular}{lllll}
\hline \multirow{2}* & \multicolumn{2}{c}{Size} &
\multicolumn{2}{c}{CPU time (sec)} \\
\cline{2-5} & Orig. BCN & Quotient  & Orig.
BCN & Quotient \\
\hline $m = 1$ & \multirow{2}*{$32768$} & \multirow{2}*{$6815$}
& \multirow{2}*{$9256.08$} & \multirow{2}*{$454.47$} \\
$k = 1$ \\
\hline $m = 1$ & \multirow{2}*{$32768$} & \multirow{2}*{$5807$}
& \multirow{2}*{$9845.81$} & \multirow{2}*{$475.89$} \\
$k = 100$ \\
\hline $m = 2$ & \multirow{2}*{$16384$} & \multirow{2}*{$4014$}
& \multirow{2}*{$2327.27$} & \multirow{2}*{$303.04$} \\
$k = 1$ \\
\hline $m = 2$ & \multirow{2}*{$16384$} & \multirow{2}*{$4647$} &
\multirow{2}*{$2401.16$} & \multirow{2}*{$158.76$} \\
$k = 100$ \\
\hline $m = 3$ & \multirow{2}*{$8192$} & \multirow{2}*{$2793$} &
\multirow{2}*{$590.42$} & \multirow{2}*{$53.79$} \\
$k = 1$ \\
\hline $m = 3$ & \multirow{2}*{$8192$} & \multirow{2}*{$3071$} &
\multirow{2}*{$601.22$} & \multirow{2}*{$65.40$} \\
$k = 100$ \\
\hline $m = 5$ & \multirow{2}*{$2048$} & \multirow{2}*{$887$} &
\multirow{2}*{$37.78$} & \multirow{2}*{$5.50$} \\
$k = 1$ \\
\hline $m = 5$ & \multirow{2}*{$2048$} & \multirow{2}*{$1015$} &
\multirow{2}*{$32.20$} & \multirow{2}*{$6.28$} \\
$k = 100$ \\
\hline
\end{tabular}
\end{table}
The second and third columns give the number of states of the
original networks and the number of states of the quotient systems,
reflecting the degree of reduction. The fourth column records the
CPU time spent for constructing stabilizing controllers directly
based on the original networks. Specifically, we followed the design
procedure proposed by Fornasini and Valcher \cite{fornasini2013b}
and Li et al. \cite{r.li2013} when $k = 1$, and the procedure of Guo
et al. \cite{guo2015} when $k = 100$. The CPU time required for
determining stabilizers via Proposition~\ref{prop2} is shown in the
last column. Similarly, Table~\ref{tab3} compares the network size
and the CPU time to obtain a solution to Problem~\ref{plm1}, with $T
= 40$.
\begin{table}
\caption{Comparison between direct and quotient-based methods for
solving Problem~\ref{plm1}.} \label{tab3} \centering
\begin{tabular}{lllll}
\hline \multirow{2}* & \multicolumn{2}{c}{Size} &
\multicolumn{2}{c}{CPU time (sec)} \\
\cline{2-5} & Orig. BCN & Quotient  & Orig.
BCN & Quotient \\
\hline $m = 1$ & $32768$ & $6087$ & $490.67$ & $115.87$ \\
$m = 2$ & $16384$ & $4506$ & $256.51$ & $74.94$ \\
$m = 3$ & $8192$ & $3441$ & $130.33$ & $55.50$ \\
$m = 5$ & $2048$ & $829$ & $27.96$ & $10.81$ \\
\hline
\end{tabular}
\end{table}
For the sake of simplicity, we assumed that the function $l(u,x)$
depends only on $u$, with the value $1$ if $u_1 = 1$ and $0$ if $u_1
= 0$; the function $g(x)$ was assumed to take the value $5$ if $x_1
= 0$ and the value $0$ otherwise. (Here we use binary
representations of $x$ and $u$.) The corresponding optimal control
problem was solved both by applying the algorithm of Fornasini and
Valcher \cite{fornasini2014} directly to the original network, and
by using the indirect method given in Proposition~\ref{prop3}. It is
seen that the proposed methods offer a reduction in computation time
compared to the state of the art, and the extent of reduction
increases (as a trend) with increasing size of the original network.
All computations were run on an Intel Core i7-3.00 GHz personal
computer with $8$ GB of RAM.

\section{A biological example}

\label{sec5}

We apply our methods to a Boolean model for lactose metabolism in
the bacterium \textit{E. coli} \cite{veliz2011}. The model consists
of $13$ variables ($1$ mRNA, $5$ proteins, and $7$ sugars) denoted
by $M$, $P$, $B$, $C$, $R$, $R_m$, $A$, $A_m$, $L$, $L_m$, $L_e$,
$L_{em}$ and $G_e$. Here, $R$ and $R_m$ are combined to indicate
concentration levels of a specific substance (the repressor
protein); that is, the concentration is low when $(R, R_m) = (0,
0)$, medium when $(R, R_m) = (0, 1)$, and high when $(R, R_m) = (1,
1)$. The fourth possibility, $(R, R_m) = (1, 0)$, is meaningless and
not allowed. The same situation is for the pairs $(A, A_m)$, $(L,
L_m)$, and $(L_e, L_{em})$ (see \cite{veliz2011} for more details on
this aspect). The equations describing the model are as follows:
\begin{align} \label{eq15a}
&M(t+1) = C(t) \wedge \neg R(t) \wedge \neg R_m (t) , \notag \\
&P(t+1) = M(t) , \qquad B(t+1) = M(t), \notag \\
&C(t+1) = \neg G_e (t), \notag \\
&R(t+1) = \neg A(t) \wedge \neg A_m (t) , \notag \\
&R_m (t+1) = (\neg A(t) \wedge \neg A_m (t)) \vee R(t) , \\
&A(t+1) = B(t) \wedge L(t) , \notag \\
&A_m (t+1) = L(t) \vee L_m (t) , \notag \\
&L(t+1) = P(t) \wedge L_e(t) \wedge \neg G_e(t) , \notag \\
&L_m (t+1) = ((L_{em}(t) \wedge P(t)) \vee L_e(t)) \wedge \neg
G_e(t) . \notag
\end{align}
We assume that the concentration of extracellular lactose is low
($L_e = L_{em} = 0$), and treat the extracellular glucose levels
($G_e$) as input to the model. Then the model can be rewritten as in
(\ref{eq3}) with\footnote{Here, $N$ is not a power of $2$, since for
some Boolean pairs in the model only three of the four values are
admissible. More precisely, since each of the variables $M$, $P$,
$B$, and $C$ has two possible values, whereas each of the pairs $(R,
R_m)$, $(A, A_m)$, and $(L, L_m)$ takes on only three possible
values, the total number of states of (\ref{eq15a}) is equal to $2^4
\cdot 3^3 = 432$; thus $N = 432$ in the algebraic representation.}
$N = 432$ and $M = 2$. The matrix $F \in \mathcal{L}^{432 \times
864}$ is detailed in the Appendix.

\textit{(1) Stabilization.} When extracellular lactose levels get
low, the model is known to exhibit two steady states
\cite{veliz2011}, expressed in the canonical vector form as
$\delta_{432}^{387}$ and $\delta_{432}^{414}$. Let $\mathcal{M} =
\{\delta_{432}^{387}\}$ and let $\mathcal{S}$ be the equivalence
relation produced by the partition $\{\, \mathcal{M}, \:
\Delta_{432} - \mathcal{M} \,\}$. Then by following the procedure
described in Section~\ref{sec3}, we get a quotient system
$\Sigma_{\mathcal{R}} \colon \, x_{\mathcal{R}}(t+1) = \widetilde{F}
\ltimes u(t) \ltimes x_{\mathcal{R}}(t)$, with $x_{\mathcal{R}} \in
\Delta_8$, $u \in \Delta_2$, and $\widetilde{F} \in \mathcal{L}^{8
\times 16}$ given by
\begin{equation*}
\widetilde{F} = \big[
\begin{matrix}
\delta_8^2 & \delta_8^2 & \delta_8^7 & \delta_8^2 & \delta_8^4 &
\delta_8^7 & \delta_8^2 & \delta_8^4 & \delta_8^1 &\delta_8^1 &
\delta_8^6 & \delta_8^6 & \delta_8^3 & \delta_8^7 & \delta_8^2 &
\delta_8^4
\end{matrix} \big] .
\end{equation*}
The matrix $C$ obtained during the procedure (which is of size $8
\times 432$ and not given explicitly) satisfies $C
\delta_{432}^{387} = \delta_8^1$. It is not hard to see that for any
\begin{equation*}
K = \big[
\begin{matrix}
\delta_2^2 \; & \delta_2^2 \; & * \; & * \; & * \; & * \; & * \; &
* \,
\end{matrix} \big]
\end{equation*}
($*$ denoting columns that can be either $\delta_2^1$ or
$\delta_2^2$), the feedback law given by $x_{\mathcal{R}} \mapsto
u(x_{\mathcal{R}}) = K x_{\mathcal{R}}$ stabilizes the quotient
system to $\delta_8^1$. Proposition~\ref{prop2} then ensures that
the original model can be globally stabilized to the state
$\delta_{432}^{387}$ via the feedback law $x \mapsto u(Cx) = KCx$. A
similar argument can be made for finding a feedback controller that
stabilizes the model to the state $\delta_{432}^{414}$; the details
are not repeated here.

\begin{remark} \label{rmk5a} \normalfont
It required about $6.5$ s to find the above controller directly
based on the procedure described in \cite{fornasini2013b} and
\cite{r.li2013}. In contrast, it took only $1.16$ s to obtain the
same stabilizer by using the quotient-based method. Thus in this
case there is an increase in speed by a factor of about $5$ to $6$
when the proposed method is employed.
\end{remark}

\textit{(2) Optimal control.} Assume that $T = 3$, the initial
condition $x(0) = \delta_{432}^{10}$, and the functions $l(u, x)$
and $g(x)$ are given by
\begin{align*}
l(\delta_2^1, x) &= 1, \quad l(\delta_2^2, x) = 2 \quad (x \in
\Delta_{432}) , \\
g(\delta_{432}^1) &= \cdots = g(\delta_{432}^{54}) = 0 , \quad
g(\delta_{432}^{55}) = \cdots \\
& \qquad \qquad \qquad \qquad \qquad \qquad \;\;\;\; =
g(\delta_{432}^{432}) = 5 .
\end{align*}
Here we remark that the states $\delta_{432}^1, \ldots,
\delta_{432}^{54}$ correspond to the \textit{lac} operon, which is
responsible for the metabolism of lactose, being ON (induced); cf.
\cite{veliz2011}. The above choice of $g(x)$ then indicates that the
operon is desired to be in an ON state after intervention. By
proceeding as in Section~\ref{sec4.2}, one can obtain a quotient
system $\Sigma_{\mathcal{R}}$ with $N = 12$, $M = 2$, and the matrix
\begin{align*}
\setcounter{MaxMatrixCols}{20} \widetilde{F} &= \big[
\begin{matrix}
\delta_{12}^7 & \delta_{12}^7 & \delta_{12}^{12} & \delta_{12}^{12}
& \delta_{12}^{12} & \delta_{12}^7 & \delta_{12}^7 &
\delta_{12}^{12} & \delta_{12}^{12} & \delta_{12}^{12} &
\delta_{12}^7 & \delta_{12}^7
\end{matrix} \\
& \qquad \
\begin{matrix}
\delta_{12}^3 & \delta_{12}^4 & \delta_{12}^{12} & \delta_{12}^8 &
\delta_{12}^9 & \delta_{12}^1 & \delta_{12}^7 & \delta_{12}^{12} &
\delta_{12}^8 & \delta_{12}^9 & \delta_{12}^1 & \delta_{12}^7
\end{matrix} \big] .
\end{align*}
The matrix $C$ satisfies $C x(0) = \delta_{12}^{11}$, and the
induced functions $l_{\mathcal{R}}$ and $g_{\mathcal{R}}$ are
defined by
\begin{align*}
&l_{\mathcal{R}}(\delta_2^1, x_{\mathcal{R}}) = 1, \quad
l_{\mathcal{R}}(\delta_2^2, x_{\mathcal{R}}) = 2 \quad
(x_{\mathcal{R}} \in \Delta_{12}) , \\
&g_{\mathcal{R}}(\delta_{12}^1) = \cdots =
g_{\mathcal{R}}(\delta_{12}^7) = 5 , \quad
g_{\mathcal{R}}(\delta_{12}^8) = \cdots \\
& \qquad \qquad \qquad \qquad \qquad \qquad \qquad \qquad \quad =
g_{\mathcal{R}}(\delta_{12}^{12}) = 0 .
\end{align*}
It is straightforward to see that the input sequence
\begin{equation*}
u^\ast (0) = u^\ast (1) = \delta_2^2 , \quad \; u^\ast (2) =
\delta_2^1
\end{equation*}
is optimal for $\Sigma_{\mathcal{R}}$, with the optimal cost
$J_{\mathcal{R}}^\ast = 5$, so it also solves the optimal control
problem for the original model, and the optimal cost is $J^\ast =
J_{\mathcal{R}}^\ast = 5$. Moreover, we see from the value of
$J^\ast$ that the optimal input indeed steers the model to an ON
state, as desired.

\begin{remark} \label{rmk5b} \normalfont
As for the time comparison, we report that it took about $1.5$ s to
solve this problem directly by the method of Fornasini and Valcher
\cite{fornasini2014}, while the above indirect procedure took only
$0.79$ s. Thus there is about $2$ times saving in speed when the
quotient-based method is employed.
\end{remark}

\section{Discussions}

\label{sec6}

The paper has considered quotients of BCNs. Two possible
applications of the quotient description have been presented in
Section~\ref{sec4}, where we have seen that the stabilization and
optimal control problems of the original BCNs can be boiled down to
those of the quotient systems. Let us mention that we have presented
only a few examples of such applications, and there are quite a few
other problems such as output tracking and observability checking
that can also be dealt with in this manner. We do not include the
details of these applications for reasons of space.

Since the number of states of the quotient $\Sigma_{\mathcal{R}}$ is
precisely the number of the equivalence classes generated by
$\mathcal{R}$, the coarser the relation $\mathcal{R}$, the smaller
is $\Sigma_{\mathcal{R}}$ and, thus, the greater is the degree of
reduction. Recall that the relation $\mathcal{R}$ is required to
satisfy (\ref{eq7}), which is related to the dynamics of $\Sigma$.
Thus, the degree of reduction is affected by the specific dynamics
of the original network. Also, since in practice different relations
are required for different applications (cf. Sections~\ref{sec4.1}
and \ref{sec4.2}), despite the same original network, the reduction
degree may still be different, depending on the specific problems to
be solved. The size of the quotient systems appearing in the
numerical experiments reported in Section~\ref{sec4.3} is about
$50$--$20\%$ when compared to the original networks. In the
biological example presented in Section~\ref{sec5}, the size of the
reduced state space is less than $3\%$ of that of the original one.

In Section~\ref{sec4.3}, we have limited the discussion to networks
with $16$ nodes, since we would like to compute the control policy
on each originally generated network and list the exact time that
the standard methods require, in order to make the comparisons. Here
we report that besides these simulations, we also tested our methods
on networks with about $20$--$23$ nodes. We observed that for most
instances, the standard methods ran out of memory whereas the
proposed methods were able to obtain a solution in a matter of
minutes to hours. We do not present the detailed numerical results
due to the limitations on the paper length.

\clearpage \onecolumn
\section*{Appendix}

The matrix $F$ for the biological model discussed in
Section~\ref{sec5} is
\begin{align*}
F = \delta_{432}& [255 \;\; 258 \;\; 261 \;\; 255 \;\; 258 \;\; 261
\;\; 246 \;\; 249 \;\; 252 \;\; 264 \;\; 267 \;\; 270 \;\; 264 \;\;
267 \;\; 270 \;\; 246 \;\; 249 \;\; 252 \;\;\;\: 48 \;\;\;\: 51
\;\;\;\: 54 \\
&\;\;\: 48 \;\;\;\: 51 \;\;\;\: 54 \;\;\;\: 30 \;\;\;\: 33 \;\;\;\:
36 \;\;255 \;\; 258 \;\; 261 \;\; 255 \;\; 258 \;\; 261 \;\; 246
\;\; 249 \;\; 252 \;\; 264 \;\; 267 \;\; 270 \;\; 264 \;\; 267 \;\;
270 \\
&\; 246 \;\; 249 \;\; 252 \;\; 264 \;\; 267 \;\; 270 \;\; 264 \;\;
267 \;\; 270 \;\; 246 \;\; 249 \;\; 252 \;\; 258 \;\; 258 \;\; 261
\;\; 258 \;\; 258 \;\; 261 \;\; 249 \;\; 249 \;\; 252 \\
&\; 267 \;\; 267 \;\; 270 \;\; 267 \;\; 267 \;\; 270 \;\; 249 \;\;
249 \;\; 252 \;\;\;\: 51 \;\;\;\: 51 \;\;\;\: 54 \;\;\;\: 51
\;\;\;\: 51 \;\;\;\: 54 \;\;\;\: 33 \;\;\;\: 33 \;\;\;\: 36 \;\; 258
\;\; 258 \;\; 261 \\
&\; 258 \;\; 258 \;\; 261 \;\; 249 \;\; 249 \;\; 252 \;\; 267 \;\;
267 \;\; 270 \;\; 267 \;\; 267 \;\; 270 \;\; 249 \;\; 249 \;\; 252
\;\; 267 \;\; 267 \;\; 270 \;\; 267 \;\; 267 \;\; 270 \\
&\; 249 \;\; 249 \;\; 252 \;\; 255 \;\; 258 \;\; 261 \;\; 255 \;\;
258 \;\; 261 \;\; 246 \;\; 249 \;\; 252 \;\; 264 \;\; 267 \;\; 270
\;\; 264 \;\; 267 \;\; 270 \;\; 246 \;\; 249 \;\; 252 \\
&\;\;\: 48 \;\;\;\: 51 \;\;\;\: 54 \;\;\;\: 48 \;\;\;\: 51 \;\;\;\:
54 \;\;\;\: 30 \;\;\;\: 33 \;\;\;\: 36 \;\; 255 \;\; 258 \;\; 261
\;\; 255 \;\; 258 \;\; 261 \;\; 246 \;\; 249 \;\; 252 \;\; 264 \;\;
267 \;\; 270 \\
&\; 264 \;\; 267 \;\; 270 \;\; 246 \;\; 249 \;\; 252 \;\; 264 \;\;
267 \;\; 270 \;\; 264 \;\; 267 \;\; 270 \;\; 246 \;\; 249 \;\; 252
\;\; 258 \;\; 258 \;\; 261 \;\; 258 \;\; 258 \;\; 261 \\
&\; 249 \;\; 249 \;\; 252 \;\; 267 \;\; 267 \;\; 270 \;\; 267 \;\;
267 \;\; 270 \;\; 249 \;\; 249 \;\; 252 \;\;\;\: 51 \;\;\;\: 51
\;\;\;\: 54 \;\;\;\: 51 \;\;\;\: 51 \;\;\;\: 54 \;\;\;\: 33 \;\;\;\:
33 \;\;\;\: 36 \\
&\; 258 \;\; 258 \;\; 261 \;\; 258 \;\; 258 \;\; 261 \;\; 249 \;\;
249 \;\; 252 \;\; 267 \;\; 267 \;\; 270 \;\; 267 \;\; 267 \;\; 270
\;\; 249 \;\; 249 \;\; 252 \;\; 267 \;\; 267 \;\; 270 \\
&\; 267 \;\; 267 \;\; 270 \;\; 249 \;\; 249 \;\; 252 \;\; 417 \;\;
420 \;\; 423 \;\; 417 \;\; 420 \;\; 423 \;\; 408 \;\; 411 \;\; 414
\;\; 426 \;\; 429 \;\; 432 \;\; 426 \;\; 429 \;\; 432 \\
&\; 408 \;\; 411 \;\; 414 \;\; 210 \;\; 213 \;\; 216 \;\; 210 \;\;
213 \;\; 216 \;\; 192 \;\; 195 \;\; 198 \;\; 417 \;\; 420 \;\; 423
\;\; 417 \;\; 420 \;\; 423 \;\; 408 \;\; 411 \;\; 414 \\
&\; 426 \;\; 429 \;\; 432 \;\; 426 \;\; 429 \;\; 432 \;\; 408 \;\;
411 \;\; 414 \;\; 426 \;\; 429 \;\; 432 \;\; 426 \;\; 429 \;\; 432
\;\; 408 \;\; 411 \;\; 414 \;\; 420 \;\; 420 \;\; 423 \\
&\; 420 \;\; 420 \;\; 423 \;\; 411 \;\; 411 \;\; 414 \;\; 429 \;\;
429 \;\; 432 \;\; 429 \;\; 429 \;\; 432 \;\; 411 \;\; 411 \;\; 414
\;\; 213 \;\; 213 \;\; 216 \;\; 213 \;\; 213 \;\; 216 \\
&\; 195 \;\; 195 \;\; 198 \;\; 420 \;\; 420 \;\; 423 \;\; 420 \;\;
420 \;\; 423 \;\; 411 \;\; 411 \;\; 414 \;\; 429 \;\; 429 \;\; 432
\;\; 429 \;\; 429 \;\; 432 \;\; 411 \;\; 411 \;\; 414 \\
&\; 429 \;\; 429 \;\; 432 \;\; 429 \;\; 429 \;\; 432 \;\; 411 \;\;
411 \;\; 414 \;\; 417 \;\; 420 \;\; 423 \;\; 417 \;\; 420 \;\; 423
\;\; 408 \;\; 411 \;\; 414 \;\; 426 \;\; 429 \;\; 432 \\
&\; 426 \;\; 429 \;\; 432 \;\; 408 \;\; 411 \;\; 414 \;\; 210 \;\;
213 \;\; 216 \;\; 210 \;\; 213 \;\; 216 \;\; 192 \;\; 195 \;\; 198
\;\; 417 \;\; 420 \;\; 423 \;\; 417 \;\; 420 \;\; 423 \\
&\; 408 \;\; 411 \;\; 414 \;\; 426 \;\; 429 \;\; 432 \;\; 426 \;\;
429 \;\; 432 \;\; 408 \;\; 411 \;\; 414 \;\; 426 \;\; 429 \;\; 432
\;\; 426 \;\; 429 \;\; 432 \;\; 408 \;\; 411 \;\; 414 \\
&\; 420 \;\; 420 \;\; 423 \;\; 420 \;\; 420 \;\; 423 \;\; 411 \;\;
411 \;\; 414 \;\; 429 \;\; 429 \;\; 432 \;\; 429 \;\; 429 \;\; 432
\;\; 411 \;\; 411 \;\; 414 \;\; 213 \;\; 213 \;\; 216 \\
&\; 213 \;\; 213 \;\; 216 \;\; 195 \;\; 195 \;\; 198 \;\; 420 \;\;
420 \;\; 423 \;\; 420 \;\; 420 \;\; 423 \;\; 411 \;\; 411 \;\; 414
\;\; 429 \;\; 429 \;\; 432 \;\; 429 \;\; 429 \;\; 432 \\
&\; 411 \;\; 411 \;\; 414 \;\; 429 \;\; 429 \;\; 432 \;\; 429 \;\;
429 \;\; 432 \;\; 411 \;\; 411 \;\; 414 \;\; 228 \;\; 231 \;\; 234
\;\; 228 \;\; 231 \;\; 234 \;\; 219 \;\; 222 \;\; 225 \\
&\; 237 \;\; 240 \;\; 243 \;\; 237 \;\; 240 \;\; 243 \;\; 219 \;\;
222 \;\; 225 \;\;\;\: 21 \;\;\;\: 24 \;\;\;\: 27 \;\;\;\: 21
\;\;\;\: 24 \;\;\;\: 27 \quad\;\; 3 \quad\;\; 6 \quad\;\; 9 \;\; 228
\;\; 231 \;\; 234 \\
&\; 228 \;\; 231 \;\; 234 \;\; 219 \;\; 222 \;\; 225 \;\; 237 \;\;
240 \;\; 243 \;\; 237 \;\; 240 \;\; 243 \;\; 219 \;\; 222 \;\; 225
\;\; 237 \;\; 240 \;\; 243 \;\; 237 \;\; 240 \;\; 243 \\
&\; 219 \;\; 222 \;\; 225 \;\; 231 \;\; 231 \;\; 234 \;\; 231 \;\;
231 \;\; 234 \;\; 222 \;\; 222 \;\; 225 \;\; 240 \;\; 240 \;\; 243
\;\; 240 \;\; 240 \;\; 243 \;\; 222 \;\; 222 \;\; 225 \\
&\;\;\: 24 \;\;\;\: 24 \;\;\;\: 27 \;\;\;\: 24 \;\;\;\: 24 \;\;\;\:
27 \quad\;\; 6 \quad\;\; 6 \quad\;\; 9 \;\; 231 \;\; 231 \;\; 234
\;\; 231 \;\; 231 \;\; 234 \;\; 222 \;\; 222 \;\; 225 \;\; 240 \;\;
240 \;\; 243 \\
&\; 240 \;\; 240 \;\; 243 \;\; 222 \;\; 222 \;\; 225 \;\; 240 \;\;
240 \;\; 243 \;\; 240 \;\; 240 \;\; 243 \;\; 222 \;\; 222 \;\; 225
\;\; 228 \;\; 231 \;\; 234 \;\; 228 \;\; 231 \;\; 234 \\
&\; 219 \;\; 222 \;\; 225 \;\; 237 \;\; 240 \;\; 243 \;\; 237 \;\;
240 \;\; 243 \;\; 219 \;\; 222 \;\; 225 \;\;\;\: 21 \;\;\;\: 24
\;\;\;\: 27 \;\;\;\: 21 \;\;\;\: 24 \;\;\;\: 27 \quad\;\; 3
\quad\;\; 6 \quad\;\; 9 \\
&\; 228 \;\; 231 \;\; 234 \;\; 228 \;\; 231 \;\; 234 \;\; 219 \;\;
222 \;\; 225 \;\; 237 \;\; 240 \;\; 243 \;\; 237 \;\; 240 \;\; 243
\;\; 219 \;\; 222 \;\; 225 \;\; 237 \;\; 240 \;\; 243 \\
&\; 237 \;\; 240 \;\; 243 \;\; 219 \;\; 222 \;\; 225 \;\; 231 \;\;
231 \;\; 234 \;\; 231 \;\; 231 \;\; 234 \;\; 222 \;\; 222 \;\; 225
\;\; 240 \;\; 240 \;\; 243 \;\; 240 \;\; 240 \;\; 243 \\
&\; 222 \;\; 222 \;\; 225 \;\;\;\: 24 \;\;\;\: 24 \;\;\;\: 27
\;\;\;\: 24 \;\;\;\: 24 \;\;\;\: 27 \quad\;\; 6 \quad\;\; 6
\quad\;\; 9 \;\; 231 \;\; 231 \;\; 234 \;\; 231 \;\; 231 \;\; 234
\;\; 222 \;\; 222 \;\; 225 \\
&\; 240 \;\; 240 \;\; 243 \;\; 240 \;\; 240 \;\; 243 \;\; 222 \;\;
222 \;\; 225 \;\; 240 \;\; 240 \;\; 243 \;\; 240 \;\; 240 \;\; 243
\;\; 222 \;\; 222 \;\; 225 \;\; 390 \;\; 393 \;\; 396 \\
&\; 390 \;\; 393 \;\; 396 \;\; 381 \;\; 384 \;\; 387 \;\; 399 \;\;
402 \;\; 405 \;\; 399 \;\; 402 \;\; 405 \;\; 381 \;\; 384 \;\; 387
\;\; 183 \;\; 186 \;\; 189 \;\; 183 \;\; 186 \;\; 189 \\
&\; 165 \;\; 168 \;\; 171 \;\; 390 \;\; 393 \;\; 396 \;\; 390 \;\;
393 \;\; 396 \;\; 381 \;\; 384 \;\; 387 \;\; 399 \;\; 402 \;\; 405
\;\; 399 \;\; 402 \;\; 405 \;\; 381 \;\; 384 \;\; 387 \\
&\; 399 \;\; 402 \;\; 405 \;\; 399 \;\; 402 \;\; 405 \;\; 381 \;\;
384 \;\; 387 \;\; 393 \;\; 393 \;\; 396 \;\; 393 \;\; 393 \;\; 396
\;\; 384 \;\; 384 \;\; 387 \;\; 402 \;\; 402 \;\; 405 \\
&\; 402 \;\; 402 \;\; 405 \;\; 384 \;\; 384 \;\; 387 \;\; 186 \;\;
186 \;\; 189 \;\; 186 \;\; 186 \;\; 189 \;\; 168 \;\; 168 \;\; 171
\;\; 393 \;\; 393 \;\; 396 \;\; 393 \;\; 393 \;\; 396 \\
&\; 384 \;\; 384 \;\; 387 \;\; 402 \;\; 402 \;\; 405 \;\; 402 \;\;
402 \;\; 405 \;\; 384 \;\; 384 \;\; 387 \;\; 402 \;\; 402 \;\; 405
\;\; 402 \;\; 402 \;\; 405 \;\; 384 \;\; 384 \;\; 387 \\
&\; 390 \;\; 393 \;\; 396 \;\; 390 \;\; 393 \;\; 396 \;\; 381 \;\;
384 \;\; 387 \;\; 399 \;\; 402 \;\; 405 \;\; 399 \;\; 402 \;\; 405
\;\; 381 \;\; 384 \;\; 387 \;\; 183 \;\; 186 \;\; 189 \\
&\; 183 \;\; 186 \;\; 189 \;\; 165 \;\; 168 \;\; 171 \;\; 390 \;\;
393 \;\; 396 \;\; 390 \;\; 393 \;\; 396 \;\; 381 \;\; 384 \;\; 387
\;\; 399 \;\; 402 \;\; 405 \;\; 399 \;\; 402 \;\; 405 \\
&\; 381 \;\; 384 \;\; 387 \;\; 399 \;\; 402 \;\; 405 \;\; 399 \;\;
402 \;\; 405 \;\; 381 \;\; 384 \;\; 387 \;\; 393 \;\; 393 \;\; 396
\;\; 393 \;\; 393 \;\; 396 \;\; 384 \;\; 384 \;\; 387 \\
&\; 402 \;\; 402 \;\; 405 \;\; 402 \;\; 402 \;\; 405 \;\; 384 \;\;
384 \;\; 387 \;\; 186 \;\; 186 \;\; 189 \;\; 186 \;\; 186 \;\; 189
\;\; 168 \;\; 168 \;\; 171 \;\; 393 \;\; 393 \;\; 396 \\
&\; 393 \;\; 393 \;\; 396 \;\; 384 \;\; 384 \;\; 387 \;\; 402 \;\;
402 \;\; 405 \;\; 402 \;\; 402 \;\; 405 \;\; 384 \;\; 384 \;\; 387
\;\; 402 \;\; 402 \;\; 405 \;\; 402 \;\; 402 \;\; 405 \\
&\; 384 \;\; 384 \;\; 387 ].
\end{align*}
The above notation means that the first column of $F$ is
$\delta_{432}^{255}$, the second column is $\delta_{432}^{258}$, and
so on.

\end{document}